\documentclass[12pt]{amsart}
\usepackage{enumerate}
\usepackage{graphicx,graphics}
\usepackage{amsfonts}
\usepackage{amssymb}
\usepackage{amsthm}
\usepackage{amsmath}
\usepackage{mathrsfs}
\usepackage{hyperref}
\input{xy}
\xyoption{all}

\newtheorem{theorem}{Theorem}[section]
\newtheorem{lemma}[theorem]{Lemma}

\newtheorem{remark}[theorem]{Remark}
\newtheorem{corollary}[theorem]{Corollary}

\newtheorem{proposition}[theorem]{Proposition}

\def\K{\mathcal{K}}
\def\R{\mathbb{R}}
\def\B{\mathbb{B}}
\def\s{\mathbb{S}}

\DeclareMathOperator{\Aff}{Aff} \DeclareMathOperator{\aff}{aff}

\DeclareMathOperator{\conv}{conv} \DeclareMathOperator{\Sim}{Sim}
 \DeclareMathOperator{\Dil}{Dil}
 
\numberwithin{equation}{section}

\begin{document}

\title{A pseudometric invariant under similarities in the hyperspace of non-degenerated compact convex sets of $\mathbb R^n$}

\author{Bernardo Gonz\'alez Merino and Natalia Jonard-P\'erez }

\email{(B.\,Gonz\'alez Merino) bg.merino@tum.de}

\email{(N.\,Jonard-P\'erez) nat@ciencias.unam.mx}



\thanks{The first author was partially supported by MINECO/FEDER project reference MTM2012-34037, Spain.
The second author has been supported by `Programa de becas
posdoctorales de la UNAM' and by  CONACYT (Mexico) under grant
204028.}

\begin{abstract}
In this work we define a new pseudometric in $\mathcal K^n_*$, the
hyperspace of all non-degenerated compact convex sets of $\mathbb
R^n$, which is invariant under similarities.    We will prove that
the quotient space generated by this pseudometric (which is the
orbit space generated by the natural action of the group of
similarities on $\mathcal K^n_*$) is homeomorphic to the
Banach-Mazur compactum $BM(n)$, while $\mathcal K^n_*$ is
homeomorphic to the topological product $Q\times\mathbb R^{n+1}$,
where $Q$ stands for the Hilbert cube.  Finally we will show some
consequences in convex geometry, namely, we measure how much two
convex bodies differ (by means of our new pseudometric) in terms of
some classical functionals.
\end{abstract}
\maketitle\markboth{B. Gonz\'alez Merino and N. Jonard-P\'erez} {A
pseudometric invariant under similarities for compact convex sets}

\section{Introduction}

The most common way to measure the distance between two non-empty closed subsets of a metric space is by means of the well known Hausdorff distance. Namely, if $A$ and $B$ are closed subsets of the metric space $(X,d)$, the Hausdorff distance between $A$ and $B$ is defined by the rule
$$d_H(A,B)=\max\left\{\sup_{a\in A}\{d(a,B)\}, ~\sup_{b\in B}\{d(b,A)\}\right\},$$
where $d(a,B)=\inf\{d(a,b)\mid b\in B\}$. However, this distance does not tell us much information about how much $A$ and $B$ are geometrically alike.
In that sense, there are some other ways to measure the distance between (classes of) closed sets.
For example, the Gromov-Hausdorff distance $d_{GH}$ between two compact metric spaces $X$ and $Y$ 
 is defined as the infimum of all Hausdorff distances $d_H(j(X), i(Y))$, where $j:X\to Z$ and $i:Y\to Z$ are isometric embeddings into a common metric space $Z$,  $d_H$ is the Hausdorff distance determined by the metric in $Z$ and the infimum is taken over all possible $Z$, $j$, and $i$.  (see e.g. \cite{Burago}). If $A$ and $B$ are isometric compact spaces, the Gromov-Hausdorff distance between them is always zero (and reciprocally). In general, $d_{GH}$   measures how far two metric spaces are from being isometric, but this is not very helpful when we want to measure the difference between two compact spaces with respect to other geometric qualities rather than the isometries.

Take into account the following more particular situation. Suppose that we have a continuous action of a topological group  $G$ on a metric space $(X,d)$. This action induces a continuous action on the hyperspace $(\mathcal C(X), d_H)$ of all non-empty compact subsets of $X$ via the formula:
$$(g, A)\longmapsto gA=\{ga\mid a\in A\},\quad A\in \mathcal C(X),~g\in G.$$
In this case we are interested in finding a useful pseudometric $\rho$ in $\mathcal C(X)$ such that
\begin{equation}\label{f:1}
 \rho(A,B)=0\quad \text{ if and only if }\quad A=gB\quad \text{ for some }\quad g\in G.
\end{equation}
\begin{equation}\label{f:2}
\rho(gA,hB)=\rho(A,B)\quad \text{ for every}\quad g, h\in G.
\end{equation}

One way to achieve this consists in finding a metric $\delta$ in the
orbit space $\mathcal C(X)/G=\{G(A)\mid A\in \mathcal C(X)\}$, where
$G(A)=\{gA\mid g\in G\}$ denotes the $G$-orbit of $A$. In this case,
the function $\rho$ defined by the rule $\rho(A,B)=\delta(G(A),
G(B))$ will satisfy the desired conditions.


Another well-known example which is closer to our interest is
the Banach-Mazur  distance between convex sets. Consider the set
$\mathcal K^n_0$ of all convex bodies of $\mathbb R^n$ (i.e., all
compact and convex subsets of $\mathbb R^n$ with non-empty interior)
equipped with the natural action of the group $\Aff(n)$ of all
invertible affine transformations of $\mathbb R^n$.  The extended
Banach-Mazur distance (or Minkowski distance)  in $\mathcal K^n_0$
is defined as
$$d_{BM}(A,B)=\inf\{\alpha\geq 1\mid A\subset gB\subset \alpha A+x,~~g\in \Aff(n), ~x\in\mathbb R^n\}.$$
It is well known that the Banach-Mazur distance satisfies the following properties
\begin{enumerate}[i)]
\item $d_{BM}(A,B)\geq 1$ and $d_{BM}(A,B)=1$ iff $A=gB$ for some $g\in \Aff(n)$,
\item $d_{BM}(A,B)=d_{BM}(B,A)$,
\item $d_{BM}(A,B)\leq d_{BM}(A,C)\cdot d_{BM}(C,B)$.
\end{enumerate}
Therefore,  by taking the logarithm of $d_{BM}$ we can define a
pseudometric in $\mathcal K^n_0$ which measures how far two convex
bodies are from belonging to the same $\Aff(n)$-orbit. If we equip
the orbit space $\mathcal K^n_{0}/\Aff(n)$ with the metric induced
by $\ln(d_{BM})$ (which in fact determines the quotient topology generated by the Hausdorff distance topology of $\mathcal{K}^n_0$),
we obtain a compact metric space known as the Banach-Mazur compactum
$BM(n)$. Let us recall that originally, the Banach-Mazur
compactum $ BM(n)$ was defined as the set of isometry classes
of $n$-dimensional Banach spaces topologized  by the original
Banach-Mazur distance, which is defined as follows.
 $$d(E,F)=\ln \big(\inf \big\{\|T\|\cdot\|T^{-1}\|  ~ \big\vert ~  T:E\to F\text{ is a linear isomorphism}\big\}\big).$$

However there are certain situations in convex geometry, where
distance $d_{BM}$ is not good enough. While working on \cite{BrGo},
we computed pairs of sets $K_m, K_M$ satisfying $K_m\subset K\subset
K_M$ sharing inradius and circumradius, diameter and minimal width, for
many sets $K\in\K^2$, and such that $K_m$ (resp. $K_M$) is minimal (resp.
maximal) fulfilling this property.
This induces to think that the distance between two sets sharing the same radii can be measured, but this cannot be done in a proper way by taking the Hausdorff distance or even by a dilatation invariant version of it (c.f. \cite{Toth}),
since, e.g., a fixed triangle and its $60^\circ$ rotation would
have non-zero distance, even though we would consider them to be
equal (as all functional values we usually consider in that case
would be). On the other hand, affine invariant distance
(e.g.~distance $d_{BM}$) collapses affine classes into the same set,
thus equalizing all $n$-ellipsoids or all $n$-simplices. Precisely
because each triangle (as well as every ellipse) plays a central role
in the boundary of the Blaschke-Santal\'o diagrams (see
\cite{Br,BrGo,HC,HCS}) we avoid those measures. Consider \cite{Sch2}
for an example where comparable stability results are achieved for
an affine invariant measure (not fitting to our problem either).
A stability estimate
quantifies the deviation of a nearextremal
convex body from the extremal ones in a previously fix inequality (see \cite{Toth}).
The deviation depends on the metric used
for the convex bodies.

Part of our motivation comes from results obtained in
\cite{DPR,DPR4,Pr}, where the authors used Blaschke-Santal\'o (or
shape) diagrams to obtain results on image analysis and pattern
recognition. In particular, they study if ratios like $D(K)/R(K)$,
$r(K)/R(K)$ or $A(K)/p(K)$ (here $A(K)$ and $p(K)$ are the
area and perimeter of $K$ and the other magnitudes will be defined later in section~\S\ref{preliminares})
can be used as shape discriminants, by
checking how far those quotients are when considering similar sets
taken from their own database.

In relation with the distances between convex sets, there is a big
interest on studying the topological structure of some hyperspaces
of convex sets equipped with the Hausdorff metric. For instance it
is well known that the hyperspace $\mathcal K^n$ of all compact
convex subsets of $\mathbb R^n$ (with $n\geq 2$) is homeomorphic to
the Hilbert cube $Q=[0,1]^\infty$ with a point removed (see
\cite{Nadler}). Another example is the hyperspace $\mathcal K^{n}_0$
which is homeomorphic to $Q\times \mathbb R^{n(n+3)/2}$ (see
\cite{AntJon}).

The aim of this work consists in studying the  hyperspace $\mathcal
K^n_*$, of all non-degenerated compact convex subsets of $\mathbb
R^n$ from the point of view of the similar transformations. The
original idea of this work  was to construct a pseudometric in
$\mathcal K^n_*$ satisfying conditions (\ref{f:1}) and (\ref{f:2})
where $G=\Sim (n)$ is the group of all similarities of $\mathbb
R^n$. We do that in section \S \ref{s: metrica}. This pseudometric
measures how far  the shapes of two sets are from each other, therefore providing a
tool to compare the geometry of convex sets (improving $d_H$ and
dilatation-invariant pseudometrics) and also inducing a geometrically richer
quotient space than $\mathcal K^n_{0}/\Aff(n)$.

We will prove that the quotient space generated by this pseudometric
(which is the orbit space generated by the group of all similarities
in $\mathcal K^n_*$) is homeomorphic to the Banach-Mazur compactum,
while $\mathcal K^n_*$ is homeomorphic to $Q\times\mathbb R^{n+1}$
(\S\ref{s:U(n)}). This result answers, in a particular case, a
question made by the referee of \cite{AntJon} (\cite[Question
7.15]{AntJon}). In Section \S \ref{s:applications} we measure how
different  $K$ and $L$ are by means of our invariant under
similarities distance, in terms of $r, D$ and $R$, and thus deriving
several stability results of those functionals with respect to that distance.
Finally in section \S\ref{s:final}, we will provide another method
to generate pseudometrics in $\mathcal K^n_0$ which are invariant
under the action of other subgroups  of $\Aff(n)$.

\section{Preliminaries}\label{preliminares}

We will base most of our results on some techniques and   notions from the theory of topological transformation groups. This is why we recall here some basic definitions and results, but we  refer the reader to the monographs \cite{Bredon} and \cite{Palais} for a more extended review of the theory of $G$-spaces. .

If $G$ is a topological group and $X$ is a $G$-space, for any $x\in X$ we denote by $G_x$ the \textit{stabilizer}  of $x$, i.e., $G_x=\{g\in G \mid gx=x\}$. For a subset $S\subset X$ and a subgroup $H\subset G$, $H(S)$ denotes the \textit{$H$-saturation} of $S$, i.e., $H(S)=\{hs\mid h\in H,s\in S\}.$ If $H(S)=S$ then we say that $S$ is an \textit{$H$-invariant} set. In particular, $G(x)$ denotes the  \textit{$G$-orbit} of $x$, i.e., $G(x)=\{gx\in X\mid g\in G\}$. The set of all orbits equipped with the quotient topology is denoted by $X/G$ and is called the \textit{$G$-orbit space} of $X$ (or simply, the orbit space).

For each subgroup $H\subset G$, we denote by $X^H$ the \textit{$H$-fixed point set} which consists of all points $x\in X$ with $H\subset G_x$. Is not difficult to see that $X^{H}$ is a closed subset of $X$.

We say that a continuous map $f:X\to Y$ between two $G$-spaces is  $G$-\textit{equivariant} (or, simply, \textit{equivariant}) if
$f(gx)=gf(x)$ for every $x\in X$ and $g\in G$. On the other hand, if $f(gx)=f(x)$ for all $x\in X$ and $g\in G$, we will say that  the map $f$ is  \textit{$G$-invariant} (or \textit{invariant}).


 A $G$-space $X$ is called \textit{proper} (in the sense of Palais \cite{Palais}) if it has an open cover consisting of, so called, {\it small} sets.
 A set  $S\subset X$ is called small if  any  point $x\in X$ has  a neighborhood $V$  such that the set $\langle S, V\rangle=\{g\in G\mid gS\cap V\neq\emptyset\}$, called the transporter from $S$ to $V$,  has compact closure in $G$.


For a given topological group $G$, a metrizable $G$-space $Y$ is called a $G$-\textit{equivariant absolute neighborhood retract}  (denoted by $Y\in G$-$\mathrm{ANR}$) if for any  metrizable $G$-space $M$ containing $Y$ as an invariant closed subset, there exist an invariant neighborhood $U$ of $Y$ in $M$ and a $G$-equivariant retraction
$r:U\to Y$. If we can always take $U=M$, then we say $Y$ is a $G$-\textit{equivariant absolute retract} (denoted by $Y\in G$-$\mathrm{AR}$).
As it happens in the non-equivariant case,  if $G$ is a compact group,  any $G$-invariant open subset of a $G$-$\rm{ANR}$ is a $G$-$\rm{ANR}$, and each $G$-contractible $G$-$\rm{ANR}$ is a $G$-$\rm{AR}$ (see, e.g., \cite{Antonyan 1980}). Recall that a $G$-space $X$ is \textit{$G$-contractible} if there exists a continuous homotopy $H:X\times[0,1]\to X$ and a $G$-fixed point $x_0\in X$ such that $H(x,0)=x$, $H(x,1)=x_0$ and $H(gx,t)=gH(x,t)$ for every $x\in X$, $g\in G$ and $t\in[0,1]$. If additionally $H(x,t)= x_0$ if and only if  $x= x_0$ or $t=1$ then we say that $X$ is \textit{$G$-strictly contractible} to $x_0$.

Let $X$ be a $G$-space and suppose that $d$ is a metric (pseudometric) in $X$. If $d(gx,gy)=d(x,y)$ for every $x,y\in X$ and $g\in G$, then we will say that  $d$ is  a \textit{$G$-invariant metric} (\textit{$G$-invariant pseudometric}) or simply an invariant metric (pseudometric). If $G$ acts on a metric space $(X,d)$ in such a way that $d$ is $G$-invariant, then we say that $G$ \textit{acts isometrically}.

For every compact group $G$ acting isometrically  on a metric space $(X,d)$,
it is well known \cite[Proposition 1.1.12]{Palais}
that the quotient topology of $X/G$ is generated by the metric
\begin{equation}\label{metrica en cociente}
d^*(G(x),G(y))=\inf\limits_{g\in G}\{d(x,gy)\}, ~~~~~\,\,~~~G(x),G(y)\in X/G.
\end{equation}

In this case, it is evident that
\begin{equation}\label{desigualdad metrica invariante}
d^*(G(x),G(y))\leq d(x,y),~~\,\,~~~x,y\in X.
\end{equation}


For every subset $A\subset X$ of a topological space $X$, we will use the symbol $\partial A$ to denote the boundary of $A$ in $X$.

Given two metric spaces $(M_1, d_1)$ and $(M_2,d_2)$,  a surjective map $f:M_1\to M_2$ is called a \textit{similarity}, provided that  there exists $\lambda>0$ such that
$$d_2(f(x), f(y))=\lambda d_1(x,y),\quad\text{for all } x, y\in M_1.$$
The constant $\lambda$ is called the \textit{ratio} of $f$.

If we consider the euclidean space $(\mathbb R^n, \|\cdot\|)$, the similarities of $\mathbb R^n$ constitute  a closed subgroup of $\Aff(n)$. We will denote this group by $\Sim(n)$. It is not difficult to see that every element  $g\in \Sim(n)$ is of the form $g(x)=u+\lambda\sigma(x)$, with $u\in\mathbb R^n$, $\lambda>0$ and $\sigma\in O(n)$, where $O(n)$ denotes the  orthogonal group.

Throughout this paper, $n$ will always denote a natural number equal or greater than $2$.
As we mentioned in the introduction, we will denote by $\mathcal K^n$ the set of all compact convex sets of $\mathbb R^n$. Two important subsets of $\mathcal K^n$ will be considered:
$\mathcal K^n_0=\{A\in\mathcal K^n\mid \dim A=n \}$, and
$\mathcal K^n_*=\{A\in\mathcal K^n\mid \dim A\geq 1 \}$.
Observe that $\mathcal K_0^n$ is the family of all convex bodies of $\mathbb R^n$, while $\mathcal K_*^n$ is the family of all non-degenerated compact convex subsets of $\mathbb R^n$.

In $\mathcal K^n$ we consider the Hausdorff distance $d_H$ induced by the euclidean distance $d$ in $\mathbb R^n$. We will also consider the natural action of $\Aff(n)$ on $\mathcal K^n$  defined through the formula
\begin{equation}\label{definicion accion}
(g, A)\longmapsto gA, \quad
gA=\{ga\mid a\in A\}.
\end{equation}
for every $g\in \Aff(n)$ and $A\in \mathcal K^n$.
Observe that the restriction of this action to $O(n)\times \mathcal K^n$ defines an isometric action (with respect to the Hausdorff distance).

For any $A\in \mathcal K ^n$, let us denote by $C(A)$ the
\textit{circumball} of $A$, i.e., $C(A)$ is the unique (euclidean)
ball of minimal volume containing the set $A$. The radius of $C(A)$,
denoted by $R(A)$, is called the \textit{circumradius}. The center
of $C(A)$ is the nearest point to $A$ (with respect to the Hausdorff
distance) and it always belongs to $A$. This point  is called the
\textit{Chebyshev point} of $A$ and will be denoted  by
$\check{c}(A)$. If we consider the Hausdorff distance in $\mathcal
K^n$, the map $\check{c}:\mathcal K^n\to \mathbb R^n$ is continuous
and $\Sim(n)$-equivariant, i.e., $\check{c}(gA)=g\check{c}(A)$ for
every $g\in \Sim (n)$ (see e.g. \cite{Mozsynska}).  On the other
hand, the map $R:\mathcal K^n\to[0,\infty)$ is always continuous and
satisfies:
$$R(\lambda A)=\lambda R(A),\quad \text{for every } A\in\mathcal K^n\text{ and }\lambda\geq 0.$$
Besides, since each element $\sigma$ of the orthogonal group $O(n)$ is an isometry, we also have that $R(\sigma A)=R(A)$, i.e., $R$ is an $O(n)$-invariant map.
For any $A\in\K^n$, the \textit{inradius} of $A$, $r(A)$, is the
biggest radius of an Euclidean ball contained in $A$, the
\textit{diameter} of $A$ (denoted by $D(A)$) is the biggest (Euclidean) distance between two
different points of $A$, and the \textit{(minimal) width} of $A$ (denoted by $w(A)$) is
the smallest (Euclidean) distance between two different supporting hyperplanes of $A$.

We will denote by $e^1,\dots,e^n$ the canonical basis of the
$n$-dimensional Euclidean space $\R^n$. By $\B$ we denote the
$n$-dimensional Euclidean closed unit ball and by $\s$ the
corresponding $(n-1)$-dimensional unit sphere, i.e.,
 \begin{align*}&\mathbb B=\big\{(x_1,\dots, x_n)\in\mathbb R^{n} \  \big | \ \sum_{i_1}^nx_i^2\leq 1\big\} \quad\text{and}\\
&\mathbb S=\big\{(x_1,\dots, x_n)\in\mathbb R^{n} \ \big | \  \sum_{i_1}^nx_i^2=1\big\}.
 \end{align*}
 The usual inner product of two vectors $u, v\in\mathbb R^n$ will be denoted by $u^{\top}v$,
 and if $u\in\mathbb R^n$, $u^{\bot}=\{x\in\mathbb R^n \ \big | \ x^{\top}u=0\}$
 will denote the hyperplane through the origin which is orthogonal to $u$.

The convex hull of $A$, $\conv (A)$ (affine hull $\aff (A)$, respectively) is
the smallest convex body (affine subspace, respectively) containing $A$.
Given $x,y\in\R^n$, $[x,y]=\conv (\{x,y\})$ represents the segment of end-points $x$
and $y$.

The Hilbert cube $[0,1]^\infty$ will be denoted by $Q$.
A Hilbert cube manifold or a $Q$-manifold is a separable, metrizable space that admits an open cover, each member of which is homeomorphic to an open subset of the Hilbert cube $Q$. We refer the reader to \cite{Chapman}, \cite{Torunczyk} and \cite{Van Mill} for an in-depth look at the theory of $Q$-manifolds.
\smallskip

A closed subset $A$ of a metric space $(X,d)$  is  called a $Z$-set if the set $\{f\in C(Q,X)\mid f(Q)\cap A=\emptyset\}$ is dense in $C(Q,X)$, being $C(Q,X)$ the space  of all continuous  maps from $Q$ to $X$ endowed with the compact-open topology.
In particular, if for every $\varepsilon>0$ there exists a map $f:X\to X\setminus A$ such that $d(x,f(x))<\varepsilon$, then $A$ is a $Z$-set.

\section{A pseudometric invariant under similarities}\label{s: metrica}
 The aim of  this section is to prove the following theorem, which was the original motivation of this work.
\begin{theorem}\label{t:mainpseudometric}
There exists a continuous pseudometric $\odot$ in $\mathcal K^n_*$ which satisfies the following two conditions.
\begin{enumerate}[\rm(1)]
\item $\odot(A,B)=0$ if and only if $A=gB$ for some $g\in \Sim(n)$.
\item $\odot(gA,hB)=\odot(A,B)$ for every $g, h\in\Sim(n)$.
\end{enumerate}
\end{theorem}

Let us start by  considering the set $\mathcal{B}^n$ consisting of all compact convex sets $A\in\mathcal K^n_*$ such that $C(A)=\mathbb B$.
\begin{theorem}\label{t:propiedades basicas de U(n)}
The set $\mathcal{B}^n$ satisfies the following conditions:

\begin{enumerate}[\rm(1)]
\item For every $A\in\mathcal K^n_*$ there exists $A'\in \mathcal{B}^n$ and $g\in \Sim(n)$ such that $gA'=A$.
\item $\mathcal{B}^n$ is $O(n)$-invariant.
\item  If $A\in \mathcal{B}^n$ and $g\in \Sim (n)$ is a similarity such that $gA\in \mathcal{B}^n$ then $g\in O(n)$.
\item $\mathcal{B}^n$ is compact.
\end{enumerate}
\end{theorem}

\begin{proof}
(1) For any $A\in \mathcal K^n_*$, there exists $g\in \Sim(n)$ such that $g\mathbb B=C(A)$ or, equivalently, $\mathbb B=g^{-1}C(A)$. Let $A':=g^{-1}A$. Evidently,
$$gA'=g(g^{-1}A)=(gg^{-1})A=A.$$
On the other hand, since $C$ is $\Sim(n)$-equivariant, we also have that
$$C(A')=C(g^{-1}A)=g^{-1}C(A)=g^{-1}g(\mathbb B)=\mathbb B$$
and therefore $A'\in \mathcal{B}^n$, as desired.

(2) Observe that $g\mathbb B=\mathbb B$ for every $g\in O(n)$. Using again the fact that $C$ is $\Sim(n)$-equivariant (and thus, $O(n)$-equivariant), we get that
$$C(gA)=gC(A)=g\mathbb B=\mathbb B, \quad\text{for every  }g\in O(n), ~A\in \mathcal{B}^n.$$
This last equality implies that $gA\in \mathcal{B}^n$ for all $g\in O(n)$ and $A\in \mathcal{B}^n$ or, in other words, $\mathcal{B}^n$ is $O(n)$-invariant.

(3) If $gA\in \mathcal{B}^n$ for some $A\in \mathcal{B}^n$ and $g\in\Sim(n)$, then
$$g\mathbb B=gC(A)=C(gA)=\mathbb B.$$
Now, it is well known that this last equality is only possible if $g\in O(n)$.

(4) Let $(A_n)_{n\in\mathbb N}$ be any sequence  in $\mathcal{B}^n$. Since all $A_n$ are  contained in $\mathbb B$, we can use Blaschke Selection Theorem (see, e.g.,  \cite[Theorem 1.8.6]{Schneider})  to conclude that there exists a subsequence $(A_{n_k})_{k\in\mathbb N}$ which converges (with respect  to the Hausdorff metric) to a compact convex subset $A\in\mathcal K^n$. Now, since $C:\mathcal K^n\to\mathcal K^n$ is continuous with respect to the Hausdorff metric, we conclude that
$$C(A)=\lim_{n\to \infty}C(A_n)=\lim_{n\to\infty}\mathbb B=\mathbb B.$$
This directly implies that $\mathcal{B}^n$ is compact, and now the proof is complete.
\end{proof}

\begin{lemma}\label{l:On retraction}
There exists an $O(n)$-equivariant retraction $\varrho:\mathcal K_*^n\to \mathcal{B}^n$ such that $\varrho(A)$ lies in the $\Sim(n)$-orbit of $A$, for every $A\in \mathcal K_*^n$.

\end{lemma}

\begin{proof}
Define $\varrho:\mathcal K_*^n\to \mathcal{B}^n$ by the rule:
$$\varrho(A)=\frac{1}{R(A)}(A-\check{c}(A)), \quad A\in\mathcal K_*^n,$$
where $R(A)$ is the circumradius of $A$ and $\check{c}(A)$ is the Chebyshev point (circumcenter) of $A$.
Since $R(A)>0$ for every $A\in\mathcal K_*^n$, the map $\varrho$ is well-defined and is obviously continuous.
Furthermore, for any  $A\in \mathcal K_*^n$ we observe that
$R(\varrho(A))=1$ and $\check{c}(A)=0$ and thus $\varrho(A)\in \mathcal{B}^n$.
On the other hand, if $A\in \mathcal{B}^n$, then $R(A)=1$, $\check{c}(A)=0$ and therefore $\varrho(A)=A$. So, $\varrho$ is a retraction.

Now, if we fix  $A\in \mathcal{K}_*^n$, the map
$g_A:\mathbb R^n\to\mathbb R^n$ defined by
$$g_A(x)=\frac{1}{R(A)}(x-\check{c}(A)), \quad x\in\mathbb R^n,$$
is a similarity and thus $g_A\in \Sim(n)$. Since $\varrho(A)=g_AA$, it is clear that $\varrho(A)$ lies in the $\Sim(n)$-orbit of $A$, as desired.

Finally, since $R$ is $O(n)$-invariant and $\check{c}$ is $O(n)$-equivariant (in fact, it is $\Sim(n)$-equivariant), if $g\in O(n)$ then  we have   $R(gA)=R(A)$ and $\check{c}(gA)=g\check{c}(A)$.  Using the linearity of $g$ we get
\begin{align*}
\varrho(gA)&=\frac{1}{R(gA)}(gA-\check{c}(gA))=\frac{1}{R(A)}(gA-g\check{c}(A))\\
&=g\Big(\frac{1}{R(A)}(A-\check{c}(A))\Big)=g\varrho(A).
\end{align*}
This proves that $\varrho$ is O(n)-equivariant and now the proof is complete.

\end{proof}

It is not difficult to see that the map $\varrho$ constructed above satisfies the following two equalities:

\begin{equation}\label{f:igualdadesvarro}
\varrho(\lambda A)=\varrho(A),\quad\text{and}\quad\varrho(A+u)=\varrho(A),
\end{equation}
for every $A\in\mathcal K^n$, $u\in\mathbb R^n$, and $\lambda>0$.
As consequence, if a similarity  $g\in \Sim(n)$ is written  as $g(x)=\lambda\sigma(x)+u$ for some $\lambda>0$, $\sigma\in O(n)$ and $u\in\mathbb R^n$, we get that
\begin{equation}\label{f.varrosacasigma}
\varrho(gA)=\sigma\varrho(A).
\end{equation}

The next result will be used later in section~\ref{s:U(n)} in order to prove that $\mathcal K^n_*$ is homeomorphic to $Q\times\mathbb R^{n+1}$.

\begin{proposition}\label{p:Kigual Upor R}
The hyperspace $\mathcal K^n_*$ is homeomorphic to $\mathcal{B}^n\times \mathbb R^{n+1}$
\end{proposition}

\begin{proof}
Define $\eta:\mathcal K^n_*\to \mathcal{B}^n\times \mathbb R^n\times (0,\infty)$ by the rule
$$\eta(A)=\big(\varrho(A), \check{c}(A), R(A)\big)\quad A\in\mathcal K^n,$$
where $\varrho$ is the retraction of Lemma~\ref{l:On retraction}, $\check{c}(A)$ is the
 Chebyshev point of $A$ and $R(A)$ is the circumradius of $A$.
 It is easy to verify that  $\eta$ is a bijective map whose inverse map is given by the rule
 $$(A, x, \lambda)\rightarrow\lambda A+x.$$
The maps $\eta$ and $\eta^{-1}$ are obviously continuous, and therefore $\eta$ is a homeomorphism.
\end{proof}

Consider now the orbit spaces $\mathcal K_*^n/\Sim(n)$ and $\mathcal{B}^n/O(n)$. To simplify the notation, the class of each $A\in\mathcal K^n_*$  in $\mathcal K_*^n/\Sim(n)$ will be denoted by $[A]$ instead of $\Sim(n)(A)$. On the other hand, the elements of $\mathcal{B}^n/O(n)$ will be denoted, as usual, by $O(n)(A)$ for each $A\in \mathcal{B}^n$.

\begin{proposition}\label{p:KyU son homeo}
$\mathcal K_*^n/\Sim(n)$ is homeomorphic to $\mathcal{B}^n/O(n)$.
\end{proposition}

\begin{proof}
Let $\pi: \mathcal{B}^n\to \mathcal K_*^n/\Sim(n)$ be the restriction to $\mathcal{B}^n$ of the orbit map. By Theorem \ref{t:propiedades basicas de U(n)}-(1), the map $\pi$ is onto and therefore  $\mathcal K_*^n/\Sim(n)$ is a compact space.  On the other hand, since $\pi$ is $O(n)$-invariant, it naturally induces a continuous onto map $\varphi:\mathcal{B}^n/O(n)\to \mathcal K^n_*/\Sim(n)$.

Now, if $[A]=\varphi(O(n)(A))=\varphi(O(n)(B))=[B]$ for $A, B\in \mathcal{B}^n$,  then we can find $g\in \Sim(n)$ such that $A=gB$. By Theorem~\ref{t:propiedades basicas de U(n)}-(3), this only can happen if $g\in O(n)$, which yields that $O(n)(A)=O(n)(B)$. In other words, $\varphi$ is an injective map.

To prove that $\varphi$ is a homeomorphism, let us consider the retraction $\varrho:\mathcal K^n_*\to \mathcal{B}^n$ of Lemma~\ref{l:On retraction}. Since $\varrho$ is $O(n)$-equivariant, it induces a continuous map
$\widetilde \varrho:\mathcal K^n_*\to \mathcal{B}^n/O(n)$ by the formula
$$\widetilde \varrho(A):=O(n)\big(\varrho(A)\big).$$
Now, if $A$ and $B$ belong to the same $\Sim(n)$-orbit, then there exists $g\in \Sim(n)$ such that
$A=gB$. Suppose that $g(x)=\lambda\sigma(x)+u$  for every $x\in \mathbb R$, where $\lambda>0$, $u\in\mathbb R^n$ and $\sigma\in O(n)$. Then using equality~\ref{f.varrosacasigma} we get
\begin{align*}
\varrho(A)=\varrho(gB)=\sigma (\varrho(B)).
\end{align*}
From this fact we infer that $\widetilde \varrho(A)=\widetilde \varrho(B)$ and therefore $\widetilde \varrho$ is constant in the $\Sim (n)$-orbits. Thus, we can use the Transgression Theorem (see e.g., \cite[Chap. IV,  Theorem 3.2]{Dugundji}) to conclude that there exists a unique continuous map $\psi:\mathcal K^n_*/\Sim(n)\to \mathcal{B}^n/O(n)$ such that $\psi\circ\pi=\widetilde\varrho$.
To finish the proof simply observe that $\psi=\varphi^{-1}$, and therefore $\varphi$ is a homeomorphism.
\end{proof}

Since $O(n)$ is a compact group acting isometrically in $\mathcal{B}^n$ (with respect to the Hausdorff metric), we can use the formula (\ref{metrica en cociente}) to define a compatible metric in $\mathcal{B}^n/O(n)$ by the rule
$$d^*_H(O(n)(A),O(n)(B))=\inf\{d_H(gA,B)\mid g\in O(n)\}.$$
Now, since $\mathcal K^n_*/\Sim(n)$ and $\mathcal{B}^n/O(n)$ are homeomorphic spaces, there is an obvious way to define a compatible metric $\Theta$ in $ \mathcal K^n_*/\Sim(n)$ as follows
$$\Theta ([A], [B])=d^*_H\big(\psi([A]),\psi([B])\big),$$
where $\psi:\mathcal K^n_*/\Sim\to \mathcal{B}^n/O(n)$ is the homeomorphism constructed in the proof of Proposition~\ref{p:KyU son homeo}.

Now we have the necessary tools to prove the main theorem of this section.

\begin{proof}[Proof of Theorem~\ref{t:mainpseudometric}]
Using the metric  $\Theta$ defined above, we can define a pseudometric in $\mathcal K^n_*$ in the following natural way
$$\odot(A,B):=\Theta([A],[B]).$$
Since $\odot$ is the composition of continuous maps, it is continuous too.
Obviously, $\odot(A,B)=0$ if and only if $[A]=[B]$ which in turns proves (1).
To prove (2), simply observe that
$$\odot(gA,hB)=\Theta([gA],[hB])=\Theta([A],[B])=\odot(A,B).$$
Now the proof is complete.
\end{proof}

In section~\ref{s:applications} we will show some applications of
the pseudometric $\odot$ by providing upper bounds in terms of known
invariant under similarities geometric functionals.

\section{The topological structure of $\mathcal{B}^n/O(n)$}\label{s:U(n)}

In this section we will prove that the quotient space $\mathcal{B}^n/O(n)$ is in fact homeomorphic to the Banach-Mazur compactum, $BM(n)$. The main idea consists in proving that $\mathcal{B}^n$ is a Hilbert cube where the natural action of $O(n)$ satisfies the conditions of \cite[Theorem 3.3]{Antonyan 2007}.
For that purpose let us  consider the family $M(n)$ consisting of all compact and convex sets $A\subset\mathbb B$ such that the intersection with the unitary sphere $\mathbb S$ is non-empty. Observe that $\mathcal{B}^n$ is an $O(n)$-invariant subset of $M(n)$. In \cite[\S 4 ]{AntJon}, several properties concerning the hyperspace $M(n)$ were proved.
It is our interest to prove that $\mathcal{B}^n$ satisfies the same properties.

\begin{lemma}\label{l:strictly On contractible}
 $\mathbb B$ is the only $O(n)$-fixed point in $\mathcal{B}^n$, and $\mathcal{B}^n$ is strictly $O(n)$-contractible to $\mathbb B$.
\end{lemma}
\begin{proof}
Obviously, $\mathbb B$ is the only $O(n)$-fixed point contained in $\mathcal{B}^n$. To prove the lemma, simply consider the homotopy $H:\mathcal{B}^n\times[0,1]\to \mathcal{B}^n$ defined by
$$H(A,t)=(1-t)A+t\mathbb B.$$
Evidently
 $H(A,0)=A$ and $H(gA,t)=gH(A,t)$ for every $g\in O(n)$, $A\in \mathcal{B}^n$ and $t\in[0,1]$. Furthermore
  $H(A,t)=\mathbb B$ if and only if $t=1$ or $A=\mathbb B$. Thus $\mathcal{B}^n$ is $O(n)$-strictly contractible to $\mathbb B^n$, as desired.
\end{proof}

\begin{lemma}\label{l:On AR}
$\mathcal{B}^n$ is an $O(n)$-$\mathrm{AR}$.
\end{lemma}

\begin{proof}
By \cite[Corollary 4.8]{Antonyan 2005}  the hyperspace $\mathcal K^n$ is an $O(n)$-$\rm{AR}$. Since $\mathcal K^n_*$ is an $O(n)$-invariant open subset of $\mathcal K^n$, it follows that $\mathcal K^n_*$ is an $O(n)$-$\rm{ANR}$. Now, using the fact that $\mathcal{B}^n$ is an $O(n)$-retract of $\mathcal K^n_*$  (Lemma~\ref{l:On retraction}) we conclude that $\mathcal{B}^n$ is an $O(n)$-$\rm{ANR}$ too. Finally, since $\mathcal{B}^n$ is $O(n)$-contractible (by Lemma~\ref{l:strictly On contractible}) we infer that $\mathcal{B}^n$ is an $O(n)$-$\rm{AR}$  (see, e.g.,  \cite{Antonyan 1980}), as required.
\end{proof}

Recall that an action of a topological group $G$ on a topological space $X$ is transitive if $G(x)=X$ for every $x\in X$. In particular, the natural action of $O(n)$ on the sphere $\mathbb S$ is transitive.
According to \cite[Proposition 4.6]{AntJon}, for  each closed subgroup  $G\subset O(n)$
 that acts non-transitively on $\mathbb S$ and each $\varepsilon > 0$,  there exists a $G$-equivariant map $\chi_\varepsilon: M(n)\to M_0(n):=M(n)\setminus \{\mathbb B\}$ which is $\varepsilon$-close to the identity map of $M(n)$ (with respect to the Hausdorff distance in $M(n)$).

If we consider the restriction $\chi_\varepsilon|_{\mathcal{B}^n}$ we obtain a $G$-equivariant map
$$\chi_\varepsilon|_{\mathcal{B}^n}:\mathcal{B}^n\to {\mathcal B_0^n}:=\mathcal{B}^n\setminus\{\mathbb B\}$$
which is $\varepsilon$-close to the identity map of $\mathcal{B}^n$. Namely,
$$d_H\big(A, \chi_\varepsilon(A)\big)<\varepsilon.$$
Since $\chi_{\varepsilon}$ is  $G$-equivariant,
the restriction  $\chi_{\varepsilon}|_{{\mathcal{B}^n}^G}:{\mathcal{B}^n}^G\to {\mathcal B_0^n}^G$ is well defined and is $\varepsilon$-close to the identity map of $\mathcal{B}^n$ (recall that ${\mathcal{B}^n}^G$ denotes the $G$-fixed point set of $\mathcal B^n$).
By the same reason, $\chi_{\varepsilon}$  induces a continuous map
$\widetilde \chi_{\varepsilon}:\mathcal{B}^n/G\to {\mathcal B_0^n}/G$ defined in each $G(A)\in \mathcal{B}^n/G$ as
$$\widetilde\chi_{\varepsilon}\big(G(A)\big)=G\big(\chi_{\varepsilon}(A)\big).$$
Since $G$ acts isometrically in $\mathcal{B}^n$ (with respect to $d_H$), we can use formula~(\ref{metrica en cociente}) to define a compatible metric in $\mathcal{B}^n/G$. By inequality (\ref{desigualdad metrica invariante}), the induced map $\widetilde\chi_{\varepsilon}$ is $\varepsilon$-close to the identity map of $\mathcal{B}^n/G$.

All previous arguments imply  the following.

\begin{corollary}\label{c:zset} Let $G \subset O(n)$ be a closed subgroup that acts nontran-
sitively on $\mathbb S$. Then
\begin{enumerate}[\rm(1)]
\item The singleton $\{\mathbb B\}$ is a $Z$-set in the set of $G$-fixed points ${\mathcal{B}^n}^G.$
\item  The class of $\{\mathbb B\}$ is a $Z$-set in the $G$-orbit space $\mathcal{B}^n/G$.
\end{enumerate}

\end{corollary}

The next step before proving the main theorem of this section, consists in proving that $\mathcal{B}^n$ is homeomorphic to the Hilbert cube. For that purpose let us consider the maps $f_\varepsilon, h_{\varepsilon}:M_0(n)\to M_0(n)$ from \cite[Proposition 4.10 and Proposition 4.11]{AntJon}. Both maps are $O(n)$-invariant and $\varepsilon$-close to the identity map of $M_0(n)$. Furthermore, their images have empty intersection since, for each $A\in M_0(n)$, the intersection of $f_\varepsilon(A)$ with $\mathbb S$ has empty interior in $\mathbb S$ while $h_\varepsilon(A)$ has not.
We will use these two functions in the proof of the following theorem.

\begin{theorem}\label{t:main properties Un}
Let $G\subset O(n)$ be a closed subgroup. Then
\begin{enumerate}[\rm(1)]
\item The $G$-orbit space ${\mathcal B_0^n}/G$  is a $Q$-manifold.
\item If $G$  acts non-transitively on $\mathbb S$, then
${\mathcal{B}^n}^G$ and ${\mathcal{B}^n}/G$  are homeomorphic to the Hilbert cube $Q$.
\item In particular, if $G$ is the trivial group, $\mathcal{B}^n$ is homeomorphic to $Q$.
\end{enumerate}
\end{theorem}

\begin{proof}
Let us consider the Hausdorff metric $d_H$ in $\mathcal{B}^n$ and the induced metric $d_H^*$ in $\mathcal{B}^n/G$ (c.f. formula \ref{metrica en cociente}). Since ${\mathcal{B}^n}^G$ is a closed subset of $\mathcal{B}^n$ and $\mathcal{B}^n/G$ is a continuous image of $\mathcal{B}^n$, it follows from the compactness of $\mathcal{B}^n$ that  ${\mathcal{B}^n}^G$ and $\mathcal{B}^n/G$ are both compact metric spaces.  On the other hand, $\mathcal B_0^n/G$ is an open subset of the separable and compact space $\mathcal{B}^n/G$ and thus $\mathcal B_0^n/G$ is separable and locally compact.

By Lemma~\ref{l:On AR}, $\mathcal{B}^n$ is an $O(n)$-$\rm{AR}$ which implies that ${\mathcal B_0^n}$ is an $O(n)$-$\rm{ANR}$. Then  $\mathcal{B}^n$ is a $G$-$\rm{AR}$ and  ${\mathcal B_0^n}$ is a $G$-$\rm{ANR}$  (see, e.g., \cite{Vries}). Since $G$ is a Lie group, we conclude that $\mathcal{B}^n/G$ is an $\rm{AR}$, while  ${\mathcal B_0^n}/G$ is an $\rm{ANR}$ (see
\cite{Antonyan 88}).

(1) According to Toru\'nczyk's Characterization Theorem (see \cite[Theorem 1]{Torunczyk}), to prove that ${\mathcal B_0^n}/G$ is a $Q$-manifold it is enough to find, for every $\varepsilon>0$, continuous functions $ \widetilde f_\varepsilon, \widetilde h_\varepsilon:{\mathcal B_0^n}/G\to {\mathcal B_0^n}/G$ with disjoint images and $\varepsilon$-close to the identity map of ${\mathcal B_0^n}/G$. For that purpose, just define $\widetilde f_\varepsilon$ and $\widetilde h_\varepsilon$ by the rule:
$$\widetilde f_\varepsilon\big(G(A)\big)=G\big(f_\varepsilon(A)\big)\quad\text{ and }\quad\widetilde h_\varepsilon\big(G(A)\big)=G\big(h_\varepsilon(A)\big).$$
 These two maps are well defined, continuous and satisfy the required condition. Therefore, ${\mathcal B_0^n}/G$ is a $Q$-manifold.

(2) Observe that the class of $\{\mathbb B\}$ in $\mathcal{B}^n/G$ is a $Z$-set and coincides with the complement of ${\mathcal B_0^n}/G$ in $\mathcal{B}^n/G$. Since ${\mathcal B_0^n}/G$ is a $Q$-manifold, we conclude from \cite[\S 3]{Torunczyk} that $\mathcal{B}^n/G$ is a $Q$-manifold too.  Let us also observe that $\mathcal{B}^n/G$ is a compact $\rm{AR}$ and therefore, by \cite[Theorem 7.5.8]{Van Mill}, we infer that $\mathcal{B}^n/G$ is homeomorphic to the Hilbert cube.

Now let us prove that  ${\mathcal{B}^n}^G$ is homeomorphic to the Hilbert cube. Since $\mathcal{B}^n$ is an $O(n)$-$\rm{AR}$ it follows from \cite[Theorem 3.7]{Antonyan 2005-b} that ${\mathcal{B}^n}^G$ is an  $\rm{AR}$.

Next, for every $\varepsilon>0$, let us consider the restrictions $\phi_\varepsilon=f_{\varepsilon}|_{{\mathcal B_0^n}^G}$ and $\eta_\varepsilon=h_\varepsilon|_{{\mathcal B_0^n}^G}$ of the maps $f_\varepsilon$ and $h_\varepsilon$ defined above (c.f. \cite{AntJon}). Thus $\phi_\varepsilon$ and $\eta_\varepsilon$ are well defined continuous maps with disjoint images and both of them are $\varepsilon$-close to the identity map of ${\mathcal B_0^n}^G$. Thus, using Toru\'nczyk's Characterization Theorem  (\cite[Theorem 1]{Torunczyk}) again, we conclude that ${\mathcal B_0^n}^G$ is a $Q$-manifold. By Corollary~\ref{c:zset}, $\{\mathbb B\}$ is a $Z$-set in ${\mathcal{B}^n}^G$ and thus ${\mathcal{B}^n}^G$ is a $Q$-manifold too (\cite[\S 3]{Torunczyk}). Since ${\mathcal{B}^n}^G$ is a compact $\rm{AR}$  we infer from \cite[Theorem 7.5.8]{Van Mill} that ${\mathcal{B}^n}^G$ is homeomorphic to the Hilbert cube, as desired.

Part (3) of the theorem is an obvious consequence of part (2).
\end{proof}

Now we can bring the results of this section all together in order to conclude the following.
\begin{corollary}\label{c:properties of Un}

$\mathcal{B}^n$ is a Hilbert cube endowed with an  $O(n)$-action satisfying the following properties:
\begin{enumerate}[\rm(1)]
\item $\mathcal{B}^n$ is an $O(n)$-$\mathrm{AR}$ with a unique $O(n)$-fixed point, $\mathbb B$,
\item $\mathcal{B}^n$ is strictly $O(n)$-contractible to $\mathbb B$,
\item For a closed subgroup $G\subset O(n)$, the set ${\mathcal{B}^n}^G$ equals the singleton $\{\mathbb B\}$ if and only if $G$ acts transitively  on the unit sphere $\mathbb S$, and ${\mathcal{B}^n}^{G}$ is homeomorphic to the Hilbert cube whenever ${\mathcal{B}^n}^{G}\neq\{\mathbb B\},$
\item For any closed subgroup $G\subset O(n)$,  the $G$-orbit space ${\mathcal B_0^n}/G$ is a $Q$-manifold.
\end{enumerate}
\end{corollary}

Finally, we can combine corollary~\ref{c:properties of Un} with \cite[Theorem 3.3]{Antonyan 2007} to obtain the main result of this section.

\begin{theorem}
 $\mathcal{B}^n/O(n)$ is homeomorphic to the Banach-Mazur compactum $BM(n)$.
\end{theorem}

Since $\mathcal{B}^n/O(n)$ is homeomorphic to $\mathcal K^n_*/\Sim(n)$ (according to Proposition~\ref{p:KyU son homeo}), we also have

\begin{corollary}
 $\mathcal K_*^n/\Sim(n)$ is homeomorphic to the Banach-Mazur compactum $BM(n)$.
\end{corollary}
Finally, if we combine Theorem~\ref{t:main properties Un} with Proposition~\ref{p:KyU son homeo} we get the topological structure of $\mathcal K^n_*$.
\begin{corollary}
The hyperspace $\mathcal K^n_*$ is  a contractible $Q$-manifold homeomorphic to $Q\times \mathbb R^{n+1}$.
\end{corollary}
As we mentioned in the introduction, this last corollary answers a particular case of Question~7.15 of \cite{AntJon}.

\section{Upper bounds of $\odot(\,\cdot\,,\,\cdot\,)$}\label{s:applications}

We start this section by collecting some known results. The first of them (also known as Jung's Theorem)
was shown in \cite{Ju}, whereas the second was proved in \cite[pg. 59]{BoFe}.

\begin{theorem}\label{p:Jung}
Let $K\in\mathcal K^n$. Then
\begin{equation}\label{desigualdad circumradio diametro}
\sqrt{\frac{2(n+1)}{n}}R(K)\leq D(K).
\end{equation}

Equality holds iff $K$ contains a regular simplex of edge
length $D(K)=\sqrt{\frac{2(n+1)}{n}}R(K)$.
\end{theorem}

\begin{proposition}\label{p:charCirInr}
Let $K\in\mathcal K^n$, $c\in K$, $t\geq 0$, with $c+t\B\subset K\subset\B$. Then:
\begin{enumerate}[\rm(1)]
\item $R(K)=1$ iff
there exist $k\in\{2,\dots,n+1\}$ and $p^1,\dots,p^k\in\partial K\cap \s$ such that
$0\in \conv(\{p^1,\dots,p^k\})$.
\item $r(K)=t$ iff
there exists $k\in\{2,\dots,n+1\}$ and $u^1,\dots,u^k\in\s$ such that the
hyperplanes $H_i=c+\{x:(u^i)^{\top}x=t\}$ support $c+t\B$ and $K$ in
$c+tu^i$, $i=1,\dots,k$ and $0\in\conv(\{u^1,\dots,u^k\})$.
\end{enumerate}
\end{proposition}

First we give bounds for $\odot(K,L)$ in terms of $r/R$ and $D/R$.
Observe that inequalities in Proposition~\ref{prop:rDR} (2) - (3) can also be achieved by the positive-dilatation invariant
pseudo-metric defined in \cite{Toth}.

\begin{proposition}\label{prop:rDR}
Let $K, L\in\mathcal K^n_*$. Then:
\begin{enumerate}[\rm(1)]

\item $\odot(K,L)\leq 1$, and $\odot(K,L)=1$ only if $\dim K\neq \dim L$.

\item $\odot(K,L)\leq 2 \max\{1-r(\varrho(K)),1-r(\varrho(L))\}$.

\item If $D(\varrho(K)), D(\varrho(L))< \sqrt{\frac{2n}{n-1}}$ then
$$\odot(K,L)\leq\max\left\{1-\sqrt{1-\frac{n-1}{2n}D(\varrho(K))^2},1-\sqrt{1-\frac{n-1}{2n}D(\varrho(L))^2}\right\}$$
\end{enumerate}
If $K, L\in \mathcal{B}^n$, then $\varrho$ can be supressed from (1), (2) and (3).
\end{proposition}

\begin{proof} Since $\odot(K,L)=\odot(\varrho(L),\varrho(K))$, we can assume that $K$ and $L$ belong to $\mathcal{B}^n$.
By (1) of Proposition \ref{p:charCirInr}, $0\in K\cap L$, thus
$K\subset 0+\B\subset L+\B$ as well as $L\subset 0+\B\subset K+\B$,
from which $\odot(K,L)\leq 1$. If $\dim K=\dim L$, let $g\in O(n)$ be
such that $H:=\aff(gK)=\aff(L)$ and $r, s>0$ satisfying $r\B_H\subset gK$
and $s\B_H\subset L$, where $\B_H=\B\cap H$. Then we have the following containments:
\begin{align*}
&gK\subset s\B_H+(1-s)\B_H\subset L+(1-s)\B_H\subset L+(1-s)\B,\\
 &L\subset r\B_H+(1-r)\B_H \subset gK+(1-r)\B_H\subset gK+(1-r)\B_H,
\end{align*}
and hence $\odot(K,L)\leq\max\{1-r,1-s\}<1$.

We now show (2). Let $c\in\mathbb R^n$ be such that  $c+r(L)\B\subset L$.
Observe that $K\subset \B\subset c+r(L)\B+2(1-r(L))\B$ and
by an analogous argument one gets $\odot(K,L)\leq d_H(K,L)\leq 2\max\{1-r(K),1-r(L)\}$. Equality
holds if, for instance, $K=\B$ and $L=\{x\in\R^n:u^{\top}x\leq ||u||^2\}\cap\B$,  $||u||\leq 1$. Indeed, in this case
$r(L)=(1+||u||)/2$, $\odot(L,\B)=1-||u||$ and thus
\[
2\max\{1-1,1-r(L)\}=2(1-r(L))=2\frac{1-||u||}{2}=\odot(L,\B).
\]

We finally prove (3). Let $p^1,\dots,p^k\in K\cap\s$,
$k\in\{2,\dots,n+1\}$, be the points provided in (1) of Proposition
\ref{p:charCirInr}, and observe that
$S:=\conv(\{p^1,\dots,p^k\})\subset K$. We now compute
$s\geq 0$ such that $\B\subset S+s\B$, thus showing $L\subset\B\subset
S+s\B\subset K+s\B$.


Since (1) in Proposition \ref{p:charCirInr} implies $0\in S\subset
K$, we define $t\geq 0$ as the biggest scalar satisfying $t\B\subset
S$.  Now,  by Theorem~\ref{p:Jung}, we have that the diameter of all $(n-1)$-dimensional simplices in $\mathcal{B}^n$ ranges in the interval $[\sqrt{2n/(n-1)},2]$. Since $D(S)\leq D(K)<\sqrt{\frac{2n}{n-1}}$, we conclude that $\dim S=n$
and thus $t>0$ and $k=n+1$.

Let us suppose
without loss of generality that $F=\conv(\{p^1,\dots,p^n\})$ satisfies $t\B\cap
F\neq\emptyset$. Let $u\in\s$ be such that $\aff(F)=tu+u^{\bot}$.
Then $F\subset(\aff(F)\cap\B)=tu+\sqrt{1-t^2}\B\cap u^{\bot}$ and
$tu\in F$. Using again (1) of Proposition \ref{p:charCirInr} we infer that
$R(F)=\sqrt{1-t^2}$. Now, by Theorem \ref{p:Jung} applied to $F$ in the subspace $\aff(F)$ it
follows that
\[
\sqrt{\frac{2n}{n-1}}\sqrt{1-t^2}\leq D(S)\leq D(K).
\]
Isolating $t$ we get
\[
t\geq\sqrt{1-\frac{n-1}{2n}D(K)^2}.
\]
Since $L\subset \B\subset t\B+(1-t)\B\subset S+(1-t)\B\subset K+(1-t)\B$ we conclude that
$s\leq 1-\sqrt{1-\frac{n-1}{2n}D(K)^2}$. Using a completely analogous
argument we obtain $K\subset L+s'\B$ for $s'\leq 1-\sqrt{1-\frac{n-1}{2n}D(L)^2}$. This
implies the inequality (3).
\end{proof}

In Proposition~\ref{prop:rDR},
inequality (3) seems to be non-sharp, but in any case, if we are allowed to select $L=\B$, we find a body $K$
for which (3) attains equality. Let $K=\conv(\{u,S^u\})$, where
$||u||=1$ and $S^u$ is a $(n-1)$-regular simplex contained in
$\aff (S^u)=-tu+u^{\bot}$, $t\in[0,1/n]$ and with circumball $\aff(S^u)\cap\B$. Using the computations
inside the proof we get that $t=\sqrt{1-R(S^u)^2}$ which together with Theorem \ref{p:Jung} gives
\[
\odot(K,\B)=1-t=1-\sqrt{1-\frac{n-1}{2n}D(K)^2}.
\]

Let us also observe that
if $D(\varrho(K))\geq\sqrt{\frac{2n}{n-1}}$ we cannot use Proposition~\ref{prop:rDR}-(3)
to get a better upper bound as it is shown by taking $L=\B$ and
$K=S^{n-1}\in \mathcal{B}^n$, an $(n-1)$-dimensional simplex with diameter $D(\varrho(K))$.


In \cite{BrGo} the authors computed, for a given set $K\in {\mathcal B^2}$
with fixed $r=r(K), \omega=\omega(K)$ and $D=D(K)$,
sets $K_m$ and $K_M$ satisfying $K_m\subset K\subset K_M$
sharing all radii with $K$ and being minimum and maximum up to
suitable orthogonal transformation. Thus, these results give
inmediately sharp upper bounds of $\odot(K,L)$ in terms
of all their radii in the cases in which we studied those sets
and when both $K$ and $L$ share all radii.

We present an example in the most general case
of n-dimensional convex bodies. To do so, we recall that
for any given set $K\in \mathcal{B}^n$, $CK$ is a \textit{completion} of $K$
within $\B$ whenever $K\subseteq CK\subseteq\B$, with $D(K)=D(CK)$, and if $CK\subsetneq L$
then $D(CK)<D(L)$ for any $L\in\K^n_0$. Those completions always exist in the
Euclidean space (see Scott \cite{Sc81}). The authors in \cite{GJ} (see also
\cite[Corollary 2.11]{BrGo2}) proved
a similar result to Proposition
\ref{p:DistJung}. More particularly, they showed in Theorem 2 that for any $K\in\K^n_0$ the following inequality holds
$$d_{BM}(CK,\B)\leq(n+\sqrt{2n(n+1)})/(n+2).$$

\begin{proposition}\label{p:DistJung}
Let $K, L\in \K^n_0$ attaining equality (\ref{desigualdad circumradio diametro}). Then
\[
\odot(K,L)\leq\sqrt{\frac{2(n+1)}{n}}-\frac{1}{n}-1.
\]
\end{proposition}

\begin{proof}
Let $T^n\in \mathcal{B}^n$ be the n-dimensional regular simplex.
Since $K$ and $L$ are extremal in
Theorem \ref{p:Jung}, then $\varrho(K)$ and $\varrho(L)$ are too, and thus
we assume $K, L\in \mathcal{B}^n$. Moreover, let $g, h\in O(n)$ be such that
$g(T^n)\subseteq K$ and $h(T^n)\subseteq L$.
Considering completions $CK$ and $CL$ of $K$ and $L$, both within $\B$,
it is well known that $$D(T^n)=D(K)=D(CK)=r(CK)+R(CK)=r(CK)+1$$ (see \cite{Sa} or \cite{BrK2}). The same holds
if we substitute $K$ by $L$. Moreover, the incenter and circumcenter of
a completion in the Euclidean space coincide (and equals $0$ in this case). Therefore
\[
\odot(K,L)=\odot(g^{-1}(K),h^{-1}(L))\leq \max\{\odot(T^n,h^{-1}(CL)),\odot(T^n,g^{-1}(CK))\}.
\]

Since $T^n\subset g^{-1}(CK)$ and $T^n\subset h^{-1}(CL)$, we only have to compute the smallest $t>0$ such that
$g^{-1}(CK)$ and $h^{-1}(CL)$ are both contained in $T^n+t\B$. Since $g^{-1}(CK)$ and $h^{-1}(CL)$ share inradius $D(T^n)-1$
(see above), $t$
has to be equal or  bigger than  $r(g^{-1}(CK))-r(T^n)=D(T^n)-1-1/n=\sqrt{2(n+1)/n}-1-1/n$.

On the other side, the hyperplanes supporting the facets of  $\lambda T^{n}$, where $$\lambda:=r(g^{-1}(CK))/r(T^n),$$
also support the set $T^n+(r(g^{-1}(CK))-r(T^n))\B$. We claim that these hyperplanes
support $g^{-1}(CK)$ too. If not, let $H$ be one of these hyperplanes strictly separating
a point $p\in g^{-1}(CK)$ from $r(g^{-1}(CK))/r(T^n)T^n$, and let $q$ be the vertex of $T^n$
opposing $H$. Since $q\in T^n\subset g^{-1}(CK)$ and the distance between $H$ and $q$
is $1+r(g^{-1}(CK))=D(T^n)$, we have that $D(g^{-1}(CK)) \geq ||p-q||>D(T^n)=D(K)$, a contradiction.
Therefore $g^{-1}(CK)\subseteq T^n+(\sqrt{2(n+1)/n}-1-1/n)\B$.
Repeating this argument
for $L$ we finally derive that
$\odot(K,L)\leq \sqrt{2(n+1)/n}-1-1/n$.

The proof itself shows that equality holds when $K=T^n$ and $L=CT^n$.
\end{proof}

The next results show an improved upper bound of $\odot(K,L)$ in terms
of $r(K), r(L), R(K), R(L)$, which can be only achieved when using
the distance $\odot(\cdot,\cdot)$ but clearly not if we use the
dilatation invariant defined in \cite{Toth}.

\begin{proposition}\label{p:applicInr}
Let $K, L\in\mathcal K^n_*$. Then
\begin{equation*}
\begin{split}
\odot(K,L)\leq&\max\left\{\sqrt{(1-r(\varrho(L))+r(\varrho(K)))^2+1-r(\varrho(K))^2}-r(\varrho(L)),\right.\\
&\left.\sqrt{(1-r(\varrho(K))+r(\varrho(L)))^2+1-r(\varrho(L))^2}-r(\varrho(K))\right\}.
\end{split}
\end{equation*}
If $K, L\in \mathcal{B}^n$, then $\varrho$ can be ommited.
\end{proposition}

\begin{proof} Without loss of generalization, we assume that  $K, L\in \mathcal{B}^n$. We start by computing an orthogonal
transformation $g\in O(n)$ and the smallest $t>0$ such that
$K\subset gL+t\B$. Let $c, c'\in \R^n$ be such that $c+r(K)\B\subset
K$ and $c'+r(L)\B\subset L$. Let $\lambda_i\geq 0$ and $u^i\in\s$,
$i=1,\dots,j$, $2\leq j\leq n+1$ (with $0=\sum_{i=1}^j\lambda_iu^i$ and $c+r(K)u^i\in\partial
K$), be the scalars and vectors provided by (2) of Proposition
\ref{p:charCirInr}. Since
$0=\sum_{i=1}^j\lambda_i (c^{\top}u^i)$ then, for one of those
products, say $i=1$, it holds $c^\top u^1\leq 0$. Let $g\in O(n)$
be a rotation such that $(u^1)^{\top}(gc')=-||u^1||\,||gc'||=-||gc'||$.

Observe that the farthest distance from points in the set $\{c+x:x^{\top}u^1\leq
r(K)\}\cap\B$ to $g(c'+r(L)\B)$ is attained by the points in the intersection
$(c+r(K)u^1+(u^1)^{\bot})\cap\s$. Furthermore, this distance is even bigger
when $c^\top u^1 = 0$ and $(1-r(L))gc'/||gc'||=gc'$ (which occurs
iff $g(c'+r(L)\B)$ is tangent to $\B$). So in order to estimate $t$, let us assume that this is our case. Let $p$
be a point in $(c+r(K)u^1+(u^1)^{\bot})\cap\s$ which is at
maximal distance from $g(c'+r(L)\B)$. Let us write $p$ as the linear combination
$$p:=r(K)u^1+\sqrt{1-r(K)^2}v$$ where $v$ is a vector in $\s$  orthogonal to $u^1$. Then by Pythagorean Theorem we
get that
\begin{equation*}
\begin{split}
(r(L)+t)^2&= ||gc'-p||^2=(||gc'||+||r(K)u^1||)^2+||r(K)u^1-p||^2\\
&=(1-r(L)+r(K))^2+1-r(K)^2.
\end{split}
\end{equation*}
Thus if $t=\sqrt{(1-r(L)+r(K))^2+1-r(K)^2}-r(L)$, then
$K\subset\{x:x^{\top}u^1\leq r(K)\}\cap\B\subset g(c'+r(L)\B)+t\B\subset gL+t\B$.
By an analogous argument we obtain that $L\subset K+s\B$ for
$s=\sqrt{(1-r(K)+r(L))^2+1-r(L)^2}-r(K)$. This completes the proof.
 \end{proof}

\begin{remark}
The inequality in Proposition~\ref{p:applicInr} attains equality if, for example,
$L=\conv(\{(e^n)^{\bot}\cap\s,-te^n+(1-t)\B\})$ and
$K=\{x\in\R^n:x^{\top}e^n\leq s\}\cap\B$, with $t,s>0$ satisfying
the properties: $L\subset K$ and the segments
$T=[-te^n,se^n+\sqrt{1-s^2}e^1]$ and $[p,e^1]$ are orthogonal, where
$p=T\cap(-te^n+(1-t)\s)$. It is clear that $R(L)=1$, $r(L)=1-t$,
$R(K)=1$, $r(K)=(1+s)/2$. Since $L\subset K$, then $\odot(K,L)=c$ where $c$ is the smallest scalar $c>0$ such that $K\subset L+c\B$. Due to the properties that
$K$ and $L$  satisfy, and the considerations in the proof, the value of $c$ is  exactly $\sqrt{(1-r(L)+r(K))^2+1-r(K)^2}-r(L)$.
\end{remark}

\medskip

Let us consider $f:\K^n_*\rightarrow[0,1]^2$,
$f(K):=(r(\varrho(K)),D(\varrho(K)))$
be the 2-dimensional Blaschke-Santal\'o diagram $(r,D,R)$ (see \cite{Sa} for
the case of $\K^2_*$). Then it yields from Propositions \ref{prop:rDR} and \ref{p:applicInr}
the following result in which we give bounds for the distance $\odot(K,L)$
depending on their values $f(K)$ and $f(L)$.
\begin{corollary}\label{diagram}
Let $p,q\in f(\K^n_*)$ and $K_p,K_q\in\K^n_*$ satisfying $f(K_p)=p$
and $f(K_q)=q$, with $p=(p_1, p_2)$ and $q=(q_1, q_2)$. Then

$$
\odot(K_p,K_q)\leq\max\left\{\sqrt{(1-q_1+p_1)^2+1-p_1^2}-q_1,\sqrt{(1-p_1+q_1)^2+1-q_1^2}-p_1\right\}.$$
If additionally $p_2,q_2\leq\sqrt{\frac{n}{2(n-1)}}$, then
\begin{equation*}
\begin{split}
\odot(K_p,K_q)&\leq\min\left\{\max\left\{1-\sqrt{1-\frac{n-1}{2n}p^2_2},1-\sqrt{1-\frac{n-1}{2n}q^2_2}\right\},\right.\\
&\left.\max\bigg{\{}\sqrt{(1-q_1+p_1)^2+1-p_1^2}-q_1,\sqrt{(1-p_1+q_1)^2+1-q_1^2}-p_1\bigg{\}}\right\}.
\end{split}
\end{equation*}
\end{corollary}

The next corollary shows a uniform stability result of $r$ and $R$ in
the Euclidean ball with respect to $\odot(\cdot,\cdot)$. The proof is a  direct consequence of
 Proposition~\ref{p:applicInr}.
\begin{corollary}
Let $K, L\in\K^n_*$ be such that $1-\varepsilon\leq r(\varrho(K))$
and $1-\varepsilon\leq r(\varrho(L))$ with $\varepsilon>0$. Then
\[
\odot(K,L)\leq\sqrt{1+4\varepsilon}+\varepsilon-1.
\]
\end{corollary}


\medskip

\begin{remark}
A natural result that $\odot$ achieves is a stability result for the functional
$F(K):=(r(K),D(K),R(K))$ in the regular simplex.
This stability relies in the fact that
for any $K\in\K^n$ with $F(K)=F(T^n)$, $T^n$ being an n-dimensional regular simplex,
there exists a suitable $g\in O(n)$ such that $gK=T^n$. Namely, we can assume without loss of generality that
$K, T^n\in \mathcal{B}^n$ as $\varrho(K)$ and $\varrho(T^n)$ share radii too. Since $$D(K)=D(T^n)=\sqrt{2(n+1)/n}R(T^n)=\sqrt{2(n+1)/n}R(K),$$
$K$ attains equality in Theorem \ref{p:Jung}, and thus there exists $g\in O(n)$ such that
$T^n\subset g(K)$ and hence $r(T^n)\leq r(g(K))$. What's more, if $T^n\neq g(K)$, then
there exists a point $p\in g(K)\backslash T^n$ which, by the convexity of $g(K)$, implies that
$\conv(p,T^n)\subseteq g(K)$. It is a straighforward computation that $r(T^n)<r(\conv(p,T^n))\leq r(g(K))=r(K)$,
a contradiction. Thus $T^n=g(K)$, as we wanted to show.
\end{remark}

\section{Relations between $\odot$ and the other distances}\label{s:final}

We will start this last section by showing  another way to construct a $G$-invariant pseudometric for convex bodies, where $G$ is a closed subgroup of $\Aff(n)$ containing the group $\Dil^+$ of all positive dilatations.  Recall that a map $g:\mathbb R^n\to\mathbb R^n$ is a positive dilatation if there exists a constant $\lambda>0$ and a fixed point $u\in\mathbb R^n$ such that
$$g(x)=\lambda x+u\quad \text{ for all }x\in\mathbb R^n.$$

\begin{remark}
Let $G$ be a closed subgroup of $\Aff(n)$ containing all positive dilatations and consider the function $d_G:\mathcal K^n_0\times \mathcal K^n_0\to [1,\infty)$ defined by the formula
$$d_G(A, B)=\inf \{\alpha\geq 1\mid A\subset \phi(B)\subset \alpha A+z,\quad \phi\in G,\,z\in \mathbb R^n\}.$$
For any convex bodies $A, B$ and $C$, the function $d_G$ satisfies the following conditions:
\begin{enumerate}[\rm(1)]
\item $d_G(A,B)=1$ if and only if $A=\phi(B)$ for some $\phi\in G$.
\item $d_G(A,B)=d_G(B,A)$.
\item $d_G(A,B)\leq d_G(A,C)\cdot d_G(C,B)$.
\item $d_G(A,B)=d_G(\phi(A), \phi(B))$ for every $\phi\in G$.
\end{enumerate}
\end{remark}

\begin{proof}
All conditions are easy to verify. We just point out that part ``only if'' of  $(1)$ is supported in the fact that each orbit $G(A)$ is closed in $\mathcal K^n_0$ with respect to the Hausdorff distance.  Indeed, since the action of $\Aff(n)$  in $\mathcal{K}^n_0$ is proper (\cite[3.3]{AntJon}) and $G\subset \Aff(n)$ is closed, we conclude that $G$ acts in $\mathcal{K}^n_0$ properly too. Now, by \cite[Proposition 1.1.4]{Palais 61}, each $G$-orbit must be closed, as we have claimed.
\end{proof}

Let us remark that the previous result is just an obvious modification of the extended Banach-Mazur distance. This theorem was already observed by G. Toth in \cite{Toth} when the group $G=\Dil^+$.

\begin{corollary}
For any closed subgroup $G$ of $\Aff(n)$ containing the group of all positive dilatations, the function $\rho_G:\mathcal{K}^n_0/G\times\mathcal K_0^n/G\to [0,\infty)$ defined by the rule
$$\rho_G(A,B)=\ln (d_G(A,B))$$
is a well defined pseudometric satisfying conditions \rm{(\ref{f:1})} and \rm{(\ref{f:2})}.
\end{corollary}

An interesting case arises when we take as the group $G$ the group of all similarities $\Sim(E)$ of a Minkowsky space $E=(\mathbb R, \|\cdot\|_*)$. In this case, $\Sim(E)$ is the closed subgroup of $\Aff(n)$  generated by all isometries of $E$ and the group of all positive dilatations $\Dil^+$.
\medskip

Now, we will  discuss briefly what is the relationship between $\odot$ and
other distances.

First, observe that the easy way to compare  $\odot$ to the Hausdorff distance is throughout the obvious inequality
$$d_H(\varrho (K), \varrho (L))\geq \odot (K,L).$$
Despite this inequality,
 there is not a clear way to compare $\odot$ to $d_H$. Indeed,  if  $T$ is a segment and  $B$ is a ball, we always have $\odot (T, B)=1$.
 However, we can always choose $T$ and $B$ in such a way that the Hausdorff distance between them is arbitrarily large or arbitrarily small.

Next,  let us observe that $d_{G}(K, L)\leq d_{G'}(K,L)$ if $G'\subset G\subset \Aff(n)$.
By this we have the following inequalities:

\[
d_{BM}(K,L)\leq d_{\Sim(n)}(K,L)\leq d_{Dil^+}(K,L).
\]

In the particular case when $G=\Sim(n)$, it is interesting to ask what  the relationship between $d_{\Sim (n)}$ and $\odot$ is. To answer this question, first let us observe that $d_{\Sim(n)}$ is only defined for compact convex sets of the same dimension, while $\odot$ is defined in a larger class of compact convex sets. Taking this into account, we now show the existing relation between $\odot$ and $d_{\Sim(n)}$.

\begin{proposition}
Let $K,L\in\K^n_0$. Then
\[
\odot(K,L)+1\leq d_{\Sim(n)}(K,L)\leq
\left(1+\frac{R(K)}{r(K)}\odot(K,L)\right)
\left(1+\frac{R(L)}{r(L)}\odot(K,L)\right).
\]
\end{proposition}

\begin{proof}
Let us call $d_S=d_{\Sim(n)}(K,L)$, $x\in\R^n$ and $\sigma$ a similarity
such that
\[
K\subseteq \sigma L\subseteq x+d_S K.
\]
For the sake of clearness we assume $x=0$, and we will omit
all instances of translations (containments
hold up to translation). In particular, we will assume that $K/R(K)$ and $\sigma(L)/R(\sigma L)$ are both contained in $\mathbb B$. By definition, it holds that
\begin{equation*}
\begin{split}
\frac{K}{R(K)}&\subseteq
\frac{\sigma L}{R(\sigma L)}+\left(\frac{R(\sigma L)}{R(K)}-1\right)\frac{\sigma L}{R(\sigma L)}\\
&\subseteq\frac{\sigma L}{R(\sigma L)}+\left(\frac{R(\sigma L)}{R(K)}-1\right)\B
\end{split}
\end{equation*}
and analogously
\begin{equation*}
\begin{split}
\frac{\sigma L}{R(\sigma L)}&\subseteq
\frac{K}{R(K)}+\left(d_S\frac{R(K)}{R(\sigma L)}-1\right)\frac{K}{R(K)}\\
&\subseteq\frac{K}{R(K)}+\left(d_S\frac{R(K)}{R(\sigma L)}-1\right)\B.
\end{split}
\end{equation*}
Since $R(K)\leq R(\sigma L)\leq d_S R(K)$ it follows that
\[
\odot(K,L)\leq\max\left\{d_S\frac{R(K)}{R(\sigma L)},\frac{R(\sigma L)}{R(K)}\right\}-1\leq d_S-1.
\]

Conversely let us define $\odot=\odot(K,L)=d_H(K/R(K),(\sigma L)/R(L))$, for some
orthogonal transformation $\sigma$. It then holds that
\begin{equation*}
\frac{K}{R(K)}\subseteq\frac{\sigma L}{R(L)}+\odot\B
\subseteq\frac{1}{R(L)}\sigma L+\frac{\odot}{r(L)}\sigma L
\end{equation*}
as well as
\begin{equation*}
\frac{\sigma L}{R(L)}\subseteq\frac{K}{R(K)}+\odot\B
\subseteq\frac{1}{R(K)}K+\frac{\odot}{r(K)}K.
\end{equation*}
Denoting by
\[
\sigma'=R(K)\left(\frac{1}{R(L)}+\frac{\odot}{r(L)}\right)\sigma
\]
we then have that
\[
K\subset\sigma' L\subset R(K)R(L)\left(\frac{1}{R(L)}+\frac{\odot(K,L)}{r(L)}\right)
\left(\frac{1}{R(K)}+\frac{\odot(K,L)}{r(K)}\right)K
\]
from which the upper bound of $d_S$ follows.
\end{proof}

\textit{Acknowledgements.} We would like to thank Ren\'e Brandenberg
for many useful questions and discussions.

\end{document}